\documentclass[]{interact}

\usepackage{epstopdf}% To incorporate .eps illustrations using PDFLaTeX, etc.
\usepackage[caption=false]{subfig}% Support for small, `sub' figures and tables

\usepackage[numbers,sort&compress]{natbib}% Citation support using natbib.sty
\bibpunct[, ]{[}{]}{,}{n}{,}{,}% Citation support using natbib.sty
\theoremstyle{plain}% Theorem-like structures provided by amsthm.sty
\newtheorem{theorem}{Theorem}[section]
\newtheorem{lemma}[theorem]{Lemma}
\newtheorem{corollary}[theorem]{Corollary}

\theoremstyle{definition}
\newtheorem{definition}[theorem]{Definition}

\theoremstyle{remark}
\newtheorem{remark}{Remark}

\begin{document}

\title{Unique positive solutions to $q$-discrete equations associated with orthogonal polynomials}

\author{
\name{T. Lasic Latimer\textsuperscript{a}\thanks{CONTACT T. Lasic Latimer. Email: tlas5434@uni.sydney.edu.au. ORCiD: 0000-0001-6859-7788.}}
\affil{\textsuperscript{a}School of Mathematics and Statistics F07, The University of Sydney, NSW 2006, Australia}
}

\maketitle

\begin{abstract}
Boelen \textit{et al.} \cite{Boelen} deduced a $q$-discrete Painlev\'e equation satisfied by the recurrence coefficients of orthogonal polynomials and conjectured that the equation had a unique positive solution. We prove their conjecture and discuss related work in the area, highlighting a potentially unifying methodology.
\end{abstract}


\section{Introduction}
In this paper, we consider the initial value problem (IVP)
\begin{equation}\label{even-odd}
  \begin{cases}
    &a_{2k}(a_{2k+1}+q^{1-2k-\kappa}a_{2k}
    +q^{2}a_{2k-1}+q^{-4k+3-\kappa}a_{2k+1}a_{2k}a_{2k-1}) \\
    &\phantom{a_{2k}(a_{2k+1}+}= (1-q^{2k})q^{\kappa + 2k-1}, \\
    &a_{2k+1}(a_{2k+2}+q^{-2k}a_{2k+1}
    +q^{2}a_{2k}+q^{-4k+1-\kappa}a_{2k+2}a_{2k+1}a_{2k}) \\
    &\phantom{a_{2k}(a_{2k+1}+}= (q^{-\kappa}-q^{2k+1})q^{\kappa + 2k},\\
    & a_{k} = 0 \; \text{for}\; k< 1, 
    \end{cases}
  \end{equation}
  where $\kappa \geq 0$, $0<q<1$ and $a_k\in\mathbb R$. Boelen {\em et al.} \cite{Boelen} showed that a solution $\{a_n\}_{n=1}^{\infty}$ arises as coefficients of a 3-term recurrence relation satisfied by $q$-orthogonal polynomials and conjectured that this IVP must have a unique positive solution. We prove this result in Theorem \ref{final theorem} of this paper. \\

  Let  $\{P_n\}_{n=0}^\infty$ be monic polynomials orthogonal with respect to a weight function $w(x)$ on an interval $I$. Taking $P_0=1$, and assuming $P_{-k}\equiv 0$ for positive integer $k$, these polynomials satisfy a 3-term recurrence relation
  \[
x\,P_n(x)=P_{n+1}+b_n\,P_n(x)+a_n\,P_{n-1}(x), \quad n\ge 0 ,
\]
where $b_n$ and $a_n$ are independent of $x$. Classical cases give rise to explicit coefficients. For example,  $b_n=0$, $a_n=n$ for Hermite polynomials $He_n(x)$.\\

Orthogonal polynomials that lie outside the classical case have been extensively studied. The semi-classical case arises under certain generalisations of conditions on the weight function \cite{Shohat1939, Freud}. In such cases, the coefficient $a_n$ turns out to satisfy nonlinear IVPs, which are known not to have explicit solutions. The search for unique positive solutions of such IVPs has been carried out for  polynomials orthogonal with respect to continuous measures. In this paper, we consider polynomials orthogonal with respect to a $q$-discrete measure $|x|^{\kappa}(q^{4}x^{4};q^{4})_{\infty}d_qx$.
  
 \subsection{Brief Literature Overview}
Although recurrence relations arise from orthogonal polynomials, methods for proving the existence of unique positive solutions need not originate from orthogonal polynomials. Such methods are often specific to the problem at hand and do not easily generalise, particularly to higher dimensional discrete equations. This is highlighted by Alsulami \textit{et al.} \cite{Alsulami} who recently proved that 
\begin{equation}\label{als}
    a_{n}(\sigma_{n,1}a_{n+1}+\sigma_{n,0}a_{n}+\sigma_{n,-1}a_{n-1}) + \kappa_{n}a_{n} = l_{n} \,,
\end{equation}
has a unique positive solution, under certain constraints on $l_{n}$, $\sigma_{n,i}$ and $\kappa_{n}$. Equation \eqref{als} is a non-linear equation which does not arise from orthogonal polynomials. However, it contains 
\begin{equation}\label{intro eq 2}
a_{n}(a_{n+1}+a_{n}+a_{n-1}) + \rho a_{n} =
    \begin{cases}
            n + \kappa, &         n \, \text{even},\\
            n, &         n \, \text{odd},
    \end{cases}
\end{equation}
as a special case. Equation \eqref{intro eq 2} is satisfied by the recurrence coefficients of polynomials orthogonal with respect to the weight $|x|^{\kappa}e^{-x^{4}/4 - \rho x^{2}/2}dx$. The proof used is independent of orthogonal polynomial theory and allows for the general form of Equation \eqref{als} but does not easily generalise other types of discrete equations.\\ 

Similarly, Clarkson \textit{et al.} \cite{Clarkson2015} recently prove that the recurrence relation satisfied by polynomials orthogonal with respect to an exponential cubic weight,
\begin{subequations}\label{cubic}
\begin{eqnarray}
    a_{n}+a_{n+1} &=& b_{n}^{2}-t \,, \\
    a_{n}(b_{n}+b_{n-1}) &=& n \,,
\end{eqnarray}
\end{subequations}
has a unique positive solution. Their proof also does not recognise the discrete equation arises from orthogonal polynomials. \\ 

The method presented in this paper allows one to show the existence of a unique positive solution to a large class of recurrence relations, arising from both $q$-discrete and continuous measures. In this way we extend the results of Magnus \cite{MagnusPos} to the $q$-discrete setting. Magnus, adding to the earlier work of Bonan and Nevai \cite{Bonan1983}, proved that recurrence relations satisfied by polynomials orthogonal with respect to weights of the form $e^{P(x)}dx$, where $P(x)$ is an even polynomial with negative leading order, have a unique positive solution. The proof was restricted to smooth weights and thus could not be extended to the $q$-discrete setting. 
\subsection{Outline}
The paper is structured as follows. In Section \ref{prelim} we define notation and recall well-known results from the literature which will be used throughout the paper. In Section \ref{method} the findings of our work is presented and Theorem \ref{final theorem} is proved. We conclude with Section \ref{discussion} where we discuss how the method can be generalised to a large class of discrete equations and future work required to extend the proof further.

\section {Preliminary Work}\label{prelim}
\begin{remark}
Unless otherwise stated, any discrete equation 
\[ f_{n}(a_{n},a_{n-1},...) = 0, \qquad \text{for $n \in \mathbb{N}$} \,,  \]
will have the initial condition $a_{n} = 0$ for $n < 1$.
\end{remark}

For completeness we introduce some well known definitions from $q$-algebra. These definitions can be found in \cite{Ernst2012}.

\begin{definition}\label{q-div}
The $q$-derivative, $D_{q}$, of $f(x)$ is defined as
\begin{equation}\nonumber
    D_{q}f(x) = \frac{f(x)-f(qx)}{x(1-q)} \,.
\end{equation}
Note that $D_{q}$ satisfies the identity
\[ D_{q}(xf(x)) = qxD_{q}f(x) + f(x) \,. \]
\end{definition}

\begin{definition}
We define $[n]_{q}$ as
\[ [n]_{q} = \frac{1-q^{n}}{1-q} \,.\]
Note that
\[ D_{q}x^{n} = [n]_{q}x^{n-1} \,. \]
\end{definition}

\begin{definition}
The Pochammer symbol $(a;q)_{\infty}$ is defined as
\begin{equation*}
    (a;q)_{\infty} = \prod_{j=0}^{\infty}(1-aq^{j}) \,.
\end{equation*}
\end{definition}

\begin{definition}
The $q$ integral of $f(x)$ from -1 to 1 is defined as
\begin{equation}\nonumber
    \int_{-1}^{1}f(x)d_{q}x = (1-q)\sum_{k=0}^{\infty} (f(q^{k})+f(-q^{k}))q^{k} \,.
\end{equation}
\end{definition}
Shifting from $q$-algebra we now look to recall the famous Favard's Theorem.

\begin{theorem}[Favard's Theorem \protect{\cite[Chapter 2 Theorem 1.5]{Freud,Assche}}]\label{favard}
Suppose we are given a sequence of polynomials $\{P_n(x)\}_{n=0}^\infty$ and real numbers $\{a_{n}\}_{n=1}^{\infty}$ with $a_{n} > 0$, which satisfy 
\begin{equation}\nonumber
    x\,P_{n} = P_{n+1} + a_{n}P_{n-1},
\end{equation}
where $P_{0} = 1$, and $P_{1} = x$. Then there exists a positive measure $\mu$ on the real line, such that these polynomials are orthogonal in $L^{2}(\mu)$.
\end{theorem}

We also recall a theorem on moment determinate measures and prove a consequential lemma.
\begin{theorem} \label{determinate theorem}
Let $\mu$ be a positive measure on $\mathbb{R}^{+}:[0,\infty)$ such that
\[ 
\int_{\mathbb{R}^{+}} e^{t|x|} d\mu(x) < \infty,
\]
for where $t>0$ is a constant. Then $\mu $ is uniquely determined by the sequence of its moments, $\mu_{k} = \int_{\mathbb{R}^{+}} x^{k} d\mu(x)$, for $k$ in $\mathbb{N}\cup\{0\}$ \cite[Theorem 2]{Lin}.
\end{theorem}

\begin{lemma} \label{uniqueness}
Let $\mu$ be a positive measure on the real line such that
\[ 
\int_{\mathbb{R}} e^{tx^{2k}} d\mu(x) < \infty,
\]
for some $k$ in $\mathbb{N}$, and constant $t>0$. Then $ \int_{\mathbb{R}} x^{2} d\mu(x) $ is determined by the sub-sequence of its moments given by
\begin{equation*}
\mu_{2kj} = \int_{\mathbb{R}} x^{2kj} d\mu(x) \, \quad \text{for $j$ in $\mathbb{N}$} .
\end{equation*}
\end{lemma}

\begin{proof}
The map $f:x \rightarrow u$ is defined as
\begin{eqnarray*}
    u &=& x^{2k} \,.
\end{eqnarray*}
Thus,
\begin{equation}\nonumber
\int_{\mathbb{R}^{+}} e^{cu} d\mu_{f} (u) = \int_{\mathbb{R}} e^{cx^{2k}} d\mu (x) < \infty \,,
\end{equation}
where $\mu_{f}(u) = \mu(f^{-1}(u))$. Applying Theorem \ref{determinate theorem}, we can observe that the measure $\mu_{f}$ is uniquely determined by its moments,
\begin{equation*}
    \int_{\mathbb{R}^{+}} u^{j} d\mu_{f}(u) = \int_{\mathbb{R}} x^{2kj} d\mu (x) = \mu_{2kj} \,.
\end{equation*}
It follows that
\begin{equation*}
    \int_{\mathbb{R}} x^{2} d\mu(x) = \int_{\mathbb{R}^{+}} u^{\frac{1}{k}} d\mu_{f} (u) \,,
\end{equation*}
is determined by $\mu_{2kj}$.
\end{proof}

\section{Method}\label{method}
We aim to prove that Equation \eqref{even-odd} has a unique positive solution. However, to begin we prove Lemma \ref{q step 1} for an alternative discrete equation (which will later be shown to be equivalent).

\begin{lemma} \label{q step 1}
Suppose we have a sequence of real numbers $\{ b_{n} \}_{n=1}^{\infty}$ which satisfy
\begin{eqnarray} 
b_{n}(b_{n+1}+b_{n}+b_{n-1} + (1-q^{2})\sum^{n-2}_{j=1}b_{j}) - (1-q^{n})q^{\kappa + n-1} &=& 0, \, n \, \mathrm{even,} \nonumber \\
b_{n}(b_{n+1}+b_{n}+b_{n-1} +(1- q^{2}) \sum^{n-2}_{j=1}b_{j}) - (q^{-\kappa}-q^{n})q^{\kappa + n-1} &=& 0, \, n \, \mathrm{odd} \,.\nonumber \\ \label{even-odd2}
\end{eqnarray}
Let $\{ P_{n} \}_{n=0}^{\infty}$ be a sequence of monic polynomials defined as
\begin{equation}\label{recurrence1}
    P_{n+1} = xP_{n} - b_{n}P_{n-1} \, ,
\end{equation}
then
\begin{equation}\label{ansatz}
    xD_{q}P_{n} = [n]_{q}P_{n} + A_{n}P_{n-2} + B_{n}P_{n-4} \,,
\end{equation}
where
\begin{subequations}
\begin{eqnarray} 
A_{n} &=& \frac{q^{1-n-\kappa}}{1-q}b_{n}b_{n-1}(\sum^{n+1}_{j=1}b_{j} - q^{2}\sum^{n-3}_{j=1}b_{j}) \,,\label{Ans2}\\
B_{n} &=& \frac{q^{3-n-\kappa}}{1-q}b_{n}b_{n-1}b_{n-2}b_{n-3} \label{Ans3}\,.
\end{eqnarray}
\end{subequations}
Furthermore
\begin{eqnarray}\label{added relations}
(q^{2-n}-1)b_{n-1}+(q^{2}-1)\sum^{n-2}_{j=1}b_{j}+q^{-2n+5-\kappa}b_{n}b_{n-1}b_{n-2} &=&0 , \, n \, \text{even,} \nonumber \\
(q^{-\kappa+2-n}-1)b_{n-1}+(q^{2}-1)\sum^{n-2}_{j=1}b_{j}+q^{-2n+5-\kappa}b_{n}b_{n-1}b_{n-2} &=&0 , \, n \, \text{odd.} \nonumber \\
\end{eqnarray}
\end{lemma}
\begin{proof}
We prove by direct substitution and show that Equation \eqref{ansatz} satisfies the recurrence relation, Equation \eqref{recurrence1}. Hence that
\begin{equation}\label{dirsub}
xD_{q}(xP_{n}) = xD_{q}(P_{n+1} + b_{n}P_{n-1}) \,.
\end{equation}
By definition of the $q$-derivative the left hand side of Equation \eqref{dirsub} is given by
\begin{equation*}
xD_{q}(xP_{n}) = x(qxD_{q}P_{n} + P_{n}) \,.
\end{equation*}
We can write the above expression in terms of the basis $P_{n}$ by applying Equations \eqref{ansatz} and \eqref{recurrence1} respectively,
\begin{eqnarray*}
xD_{q}(xP_{n}) &=& xq([n]_{q}P_{n} + A_{n}P_{n-2} + B_{n}P_{n-4}) + xP_{n} \,,\\
&=& q[n]_{q}(P_{n+1} + b_{n}P_{n-1}) + qA_{n}(P_{n-1} + b_{n-2}P_{n-3}) \\ && \, + qB_{n}(P_{n-3} + b_{n-4}P_{n-5}) + P_{n+1} + b_{n}P_{n-1} \,.
\end{eqnarray*}
Similarly we can apply Equation \eqref{ansatz} to write the right hand side of Equation \eqref{dirsub} in terms of the basis $P_{n}$,
\begin{eqnarray*}
xD_{q}(P_{n+1} + b_{n}P_{n-1}) &=& [n+1]_{q}P_{n+1} + A_{n+1}P_{n-1} + B_{n+1}P_{n-3} \\ 
&& \, + b_{n}([n-1]_{q}P_{n-1} + A_{n-1}P_{n-3} + B_{n-1}P_{n-5}) \,.
\end{eqnarray*}

Equating terms and comparing coefficients of $P_{n+1}$, $P_{n-1}$, $P_{n-3}$ and $P_{n-5}$ respectively we arrive at
\begin{subequations}
\begin{eqnarray}
[n+1]_{q} &=& q[n]_{q} + 1 \,,\label{l0} \\
q[n]_{q}b_{n} + qA_{n} + b_{n} &=& A_{n+1} + b_{n}[n-1]_{q} \,,\label{l1} \\
qA_{n}b_{n-2} + qB_{n} &=& B_{n+1} + b_{n}A_{n-1} \label{l2} \,,\\
qB_{n}b_{n-4} &=& b_{n}B_{n-1} \,. \label{l3}
\end{eqnarray}
\end{subequations}
Equation \eqref{l0} is satisfied by the definition of $[n]_{q}$. Equation \eqref{l3} is satisfied by the definition of $B_{n}$. We find that Equation \eqref{l2} is satisfied by substituting in the definition of $B_{n}$ and $A_{n}$. Thus, we are left to prove that $A_{n}$ satisfies Equation \eqref{l1}. Hence that
\begin{eqnarray*}
    q[n]_{q}b_{n} - b_{n}[n-1]_{q} + b_{n}&=& \frac{q^{-n-\kappa}}{1-q}b_{n+1}b_{n}(\sum^{n+2}_{j=1}b_{j} - q^{2}\sum^{n-2}_{j=1}b_{j}) \\ && \, - q(\frac{q^{1-n-\kappa}}{1-q}b_{n}b_{n-1}(\sum^{n+1}_{j=1}b_{j} - q^{2}\sum^{n-3}_{j=1}b_{j})) \,.
\end{eqnarray*}
Substituting $[n]_{q} = \frac{1-q^{n}}{1-q}$ above, we find after algebraic manipulation that
\begin{equation*}
    q^{\kappa + 2n-1}(1-q^{2}) = b_{n+1}(\sum^{n+2}_{j=1}b_{j} - q^{2}\sum^{n-1}_{j=1}b_{j}) - q^{2} b_{n-1}(\sum^{n}_{j=1}b_{j} - q^{2}\sum^{n-3}_{j=1}b_{j})) \,.
\end{equation*}
Equation \eqref{even-odd2} indicates that for even and odd $n$
\begin{equation*}
    b_{n+1}(\sum^{n+2}_{j=1}b_{j} - q^{2}\sum^{n-1}_{j=1}b_{j}) - q^{2} b_{n-1}(\sum^{n}_{j=1}b_{j} - q^{2}\sum^{n-3}_{j=1}b_{j})) = q^{\kappa+2n-1}(1-q^{2}) \,.
\end{equation*}
Thus $A_{n}$ is a valid solution to Equation \eqref{l1}. Hence, we have proved that Equation \eqref{ansatz} is a valid expression.\\ \\
Equation \eqref{ansatz} can now be used to prove Equation \eqref{added relations}. We will prove for odd $n$, the even case follows similarly.  Expressing $P_{n}$ as $x^{n} + \alpha_{n} x^{n-2} + ...$ and equating the coefficient of $x^{n-1}$ in Equation \eqref{recurrence1} we arrive at
\begin{equation*}
    \alpha_{n+1} = \alpha_{n} - b_{n} \,.
\end{equation*}
Thus, summing over $n$,
\begin{equation*}
    \alpha_{n} = -\sum_{j=1}^{n-1}b_{j} \,.
\end{equation*}
Equating the coefficient of $x^{n-2}$ in Equation \eqref{ansatz} shows that
\begin{equation*}
    [n-2]_{q}\alpha_{n} = [n]_{q}\alpha_{n} + A_{n} \,.
\end{equation*}
Thus,
\begin{equation}\label{An1}
    A_{n} = \frac{q^{n-2}(1-q^{2})}{1-q} \sum_{j=1}^{n-1}b_{j} \,.
\end{equation}
Substituting Equation \eqref{even-odd2} (odd) into Equation \eqref{Ans2} we have
\begin{eqnarray*}
    A_{n} &=& \frac{q^{1-n-\kappa}}{1-q}b_{n}b_{n-1}(\sum^{n+1}_{j=1}b_{j} - q^{2}\sum^{n-3}_{j=1}b_{j}) \,,\\
    &=& \frac{q^{1-n-\kappa}}{1-q}(b_{n+1}b_{n}b_{n-1} + (q^{-\kappa} - q^{n})q^{\kappa+n-2}b_{n}) \,.
\end{eqnarray*}
Hence Equation \eqref{An1} gives
\begin{equation*}
    \frac{q^{1-n-\kappa}}{1-q}(b_{n+1}b_{n}b_{n-1} + (q^{-\kappa} - q^{n})q^{\kappa+n-2}b_{n}) = \frac{q^{n-2}(1-q^{2})}{1-q} \sum_{j=1}^{n-1}b_{j}\,. 
\end{equation*}
After algebraic manipulation the above equation becomes
\begin{equation*}
    q^{3-2n-\kappa}b_{n+1}b_{n}b_{n-1} + (q^{-\kappa-n+1} - 1)b_{n} = (1-q^{2}) \sum_{j=1}^{n-1}b_{j} \,.
\end{equation*}
The odd case of Equation \eqref{added relations} follows immediately. The even case is proved in a similar way.
\end{proof}

We now shift back to Equation \eqref{even-odd}. Let $\{a_{n} \}_{n=1}^{\infty}$ be a sequence of positive real numbers which satisfy Equation \eqref{even-odd}. Define the monic polynomials $ \{ P_{n} \}_{n=0}^{\infty}$ as
\begin{equation}\label{recurrence2}
    xP_{n} = P_{n+1} + a_{n}P_{n-1} \, .
\end{equation}

\begin{theorem}\label{q step 2}
The polynomials, $ \{ P_{n} \}_{n=0}^{\infty}$, defined by Equation \eqref{recurrence2} satisfy
\begin{equation}\nonumber
    xD_{q}P_{n} = [n]_{q}P_{n} + A_{n}P_{n-2} + B_{n}P_{n-4} \,,
\end{equation}
where $[n]_{q}$, $A_{n}$ and $B_{n}$ are defined in Theorem \ref{q step 1}. 
\end{theorem}

\begin{proof}
By the definition of $P_{n}$, the recurrence coefficients $\{a_{n}\}_{n=1}^{\infty} >0$ satisfy Equation \eqref{even-odd}. Taking the even component of Equation \eqref{even-odd}, we have
\begin{equation*}
(1-q^{n})q^{\kappa + n-1} = a_{n}(a_{n+1}+q^{1-n-\kappa}a_{n}+q^{2}a_{n-1}+q^{-2n+3-\kappa}a_{n+1}a_{n}a_{n-1}) \,.
\end{equation*}
Thus adding and subtracting $(1-q^{2})\sum^{n-2}_{j=1}a_{j}$ we arrive at
\begin{multline}
a_{n}(a_{n+1}+a_{n}+a_{n-1} + (1-q^{2})\sum^{n-2}_{j=1}a_{j}) - (1-q^{n})q^{\kappa + n-1} \\
 + a_{n}((q^{1-n-\kappa}-1)a_{n}+(q^{2}-1)\sum^{n-1}_{j=1}a_{j}+q^{-2n+3-\kappa}a_{n+1}a_{n}a_{n-1}) =0 \,. \label{odd}
\end{multline}
Similarly for the odd case we have
\begin{multline}
a_{n}(a_{n+1}+a_{n}+a_{n-1} + (1-q^{2})\sum^{n-2}_{j=1}a_{j}) - (q^{-\kappa}-q^{n})q^{\kappa + n-1} \\
+a_{n}((q^{1-n}-1)a_{n}+(q^{2}-1)\sum^{n-1}_{j=1}a_{j}+q^{-2n+3-\kappa}a_{n+1}a_{n}a_{n-1}) =0 \,. \label{even}
\end{multline}
We shall prove that a positive solution to Equation \eqref{even-odd2}, (denoted by $b_{n}$) is equivalent to a solution of Equation \eqref{even-odd} (denoted by $a_{n}$) given the same initial value $a_{1}>0$. By Lemma \ref{q step 1} the result then immediately follows. \\ \\ 
Assume that $a_{n} = b_{n}$ for $n\leq k$. Furthermore, assume $k$ is even. By Equation \eqref{even} we have
\begin{multline}
a_{k}(a_{k+1}+a_{k}+a_{k-1} + (1-q^{2})\sum^{k-2}_{j=1}a_{j}) - (1-q^{k})q^{\kappa + k-1} \\
+ a_{k}((q^{1-k-\kappa}-1)a_{k}+(q^{2}-1)\sum^{k-1}_{j=1}a_{j}+q^{-2k+3-\kappa}a_{k+1}a_{k}a_{k-1}) =0 \,.
\end{multline}
Let $\epsilon = a_{k+1} - b_{k+1}$. Thus,
\begin{equation*}
a_{k}\epsilon +
q^{-2k+3-\kappa}\epsilon a_{k}^{2}a_{k-1} =0 \,.
\end{equation*}
As $a_{k} > 0$ we must have $\epsilon =0 $. The same argument holds for $k$ odd. Equivalence follows by induction.
\end{proof}
By Theorem \ref{favard} there exists a positive measure $\mu$ such that the polynomials defined by Equation \eqref{recurrence2} satisfy
\begin{equation*}
    \int_{\mathbb{R}}P_{n}P_{m} d\mu (x) = h_{n}\delta_{n,m} \, .
\end{equation*}
Without loss of generality let $\mu$ be normalised such that
\[ \int_{\mathbb{R}}\mu(x) = 1 \,. \]
We will show that this measure is independent of the positive solution to Equation \eqref{even-odd}.
\begin{theorem} \label{2.4}
Let $\mu$ be a positive normalised measure over $\mathbb{R}$ such that the polynomials, $P_{n}$, defined by Equation \eqref{recurrence2}, satisfy 
\[
\int_{\mathbb{R}}P_{n}P_{m} d\mu (x)= h_{n}\delta_{n,m} \,. \]
Then
\begin{equation*}
(1-q^{n+\kappa +1}) \int x^{n} d\mu = \int x^{n+4} d\mu ,  \quad \mathrm{for} \, n \geq 0 \,.
\end{equation*} 
Hence, $\int_{\mathbb{R}}x^{4j} d\mu(x)$ is independent of the measure $\mu$ for $j \geq 0$.
\end{theorem}

\begin{proof}
Representing $x^{n}$ in the basis $P_{i}$ we have
\begin{eqnarray}
    \int xD_{q}x^{n} d\mu &=& \int xD_{q}(\sum_{i=0}^{n}\alpha_{i}P_{i})P_{0} d\mu \,,\nonumber \\
    &=& \sum_{i=1}^{n}\alpha_{i} \int xD_{q}(P_{i})P_{0} d\mu \,,\nonumber \\
    &=& \sum_{i=1}^{n}\alpha_{i} \int ([i]_{q}P_{i} + A_{i}P_{i-2} + B_{i}P_{i-4})P_{0} d\mu \,,\nonumber \\
    &=& \int (\alpha_{2}A_{2} + \alpha_{4}B_{4})P_{0}^{2} d\mu \,, \label{dq to x4}
\end{eqnarray}
where we have applied Theorem \ref{q step 2} to express $xD_{q}P_{i}$ in the basis $\{ P_{n} \}_{n=0}^{\infty}$, then the orthogonality of $P_{n}$ to arrive at the last line. Referring to Lemma \ref{q step 1},
\begin{eqnarray*}
    A_{2} &=& \frac{q^{-1-\kappa}}{1-q}a_{2}a_{1}(a_{3}+a_{2}+a_{1}) \,,\\
    B_{4} &=& \frac{q^{-1-\kappa}}{1-q}a_{4}a_{3}a_{2}a_{1} \,.
\end{eqnarray*}
By repeated application of Equation \eqref{recurrence2} we have
\begin{eqnarray*}
    x^{4}P_{4} &=& a_{4}a_{3}a_{2}a_{1}P_{0} \,,\\
    x^{4}P_{2} &=& a_{2}a_{1}(a_{3}+a_{2}+a_{1})P_{0} \,,\\
    x^{4}P_{0} &=& a_{1}(a_{2}+a_{1})P_{0} \,.
\end{eqnarray*}
where we determine $a_{1}(a_{2}+a_{1}) = (1-q^{\kappa+1})$ using Equation \eqref{even-odd} for the case $n=1$. Substituting back into Equation \eqref{dq to x4} yields
\begin{eqnarray*}
    \int xD_{q}x^{n} d\mu  &=& \int (\alpha_{2}A_{2} + \alpha_{4}B_{4})P_{0}^{2} d\mu \,,\\
    &=& \frac{q^{-1-\kappa}}{1-q} \int (\alpha_{2}x^{4}P_{2} + \alpha_{4}x^{4}P_{4})P_{0} d\mu \,,\\
    &=& \frac{q^{-1-\kappa}}{1-q} \int (\alpha_{2}x^{4}P_{2} + \alpha_{4}x^{4}P_{4} + \alpha_{0}x^{4}P_{0})P_{0} d\mu \\
    && \, - \, \frac{q^{-1-\kappa}-1}{1-q} \int \alpha_{0}P_{0}^{2} d\mu\,.\\
\end{eqnarray*}
By the orthogonality of $P_{n}$ we have
\begin{eqnarray*}
    \int xD_{q}x^{n} d\mu  &=& \frac{q^{-1-\kappa}}{1-q}\int (\alpha_{2}x^{4}P_{2} + \alpha_{4}x^{4}P_{4} + \alpha_{0}x^{4}P_{0})P_{0} d\mu \\
    && \, - \, \frac{q^{-1-\kappa}-1}{1-q} \int \alpha_{0}P_{0}^{2} d\mu\,,\\
    &=& \frac{q^{-1-\kappa}}{1-q} \int x^{4}(\sum_{i=0}^{n}\alpha_{i}P_{i})P_{0} d\mu - \frac{q^{-1-\kappa}-1}{1-q} \int (\sum_{i=0}^{n}\alpha_{i}P_{i})P_{0} d\mu \,,\\
    &=& \frac{q^{-1-\kappa}}{1-q} \int x^{n+4} d\mu - \frac{q^{-1-\kappa}-1}{1-q} \int x^{n} d\mu \,.
\end{eqnarray*}
Thus by definition of $D_{q}$,
\begin{equation*}
    [n]_{q} \int x^{n} d\mu = \frac{q^{-1-\kappa}}{1-q} \int x^{n+4} d\mu - \frac{q^{-1-\kappa}-1}{1-q} \int x^{n} d\mu \,.
\end{equation*}
After simplifying algebra this leaves
\begin{equation*}
    (1-q^{n+\kappa +1}) \int x^{n} d\mu = \int x^{n+4} d\mu \,.
\end{equation*}
\end{proof}

\begin{corollary}
Using Theorem \ref{2.4} we can deduce that the measure $\mu$ is supported in $[-1,1]$. By Theorem \ref{2.4} we have
\begin{equation*}
(1-q^{n+\kappa +1}) \int x^{n} d\mu = \int x^{n+4} d\mu ,  \quad \mathrm{for} \, n \geq 0 \,.
\end{equation*} 
Re-arranging we arrive at
\begin{equation}\label{bounded 1}
    \int x^{n}(1-x^{4}) d\mu = q^{n+\kappa+1}\int x^{n} d\mu \quad \mathrm{for} \, n \geq 0 \,.
\end{equation}
Letting $n = 2k$ be even and taking the limit as $k$ approaches infinity gives
\begin{equation}\nonumber
\lim_{k\to\infty} x^{2k} =
    \begin{cases}
            0, &         \text{if } x \in (-1,1),\\
            +\infty, &         \text{if } x \notin [-1,1].
    \end{cases}
\end{equation}
Thus, the left hand side of Equation \eqref{bounded 1} approaches negative infinity if $\mu$ is supported outside of $[-1,1]$. However, as $\mu$ is positive, the right hand side of Equation \eqref{bounded 1} is positive. Hence, $\mu$ is supported in $[-1,1]$.
\end{corollary}
We can now combine Theorem \ref{2.4} and Lemma \ref{uniqueness} to prove the main result of this paper. 

\begin{theorem}\label{final theorem}
Suppose we have a sequence of positive numbers $\{ a_{n}\}_{n=1}^{\infty}$ which satisfy the IVP Equation \eqref{even-odd}. Then 
\[ a_{1} = \int_{-1}^{1}x^{2}w(x)d_{q}x \,, \]
where $w(x)$ is defined as
\begin{equation} \label{w(x)}
    w(x) =  \frac{((qx)^{4};q^{4})_{\infty}|x|^{\kappa}}{(1-q)^{\frac{\kappa}{4}}} \,.
\end{equation}
\end{theorem}
\begin{proof}
Boelen \textit{et al.} \cite{Boelen} prove that $w(x)d_{q}x$ satisfies
Theorem \ref{2.4}. Let $\mu$ be another normalised measure which satisfies Theorem \ref{2.4}, for an arbitrary positive solution, $\{ a_{n} \}_{n=1}^{\infty}$, to Equation \eqref{even-odd}. By the monotone convergence theorem and Theorem \ref{2.4} we have
\begin{equation} \nonumber
    \int e^{c(x^{4})} d\mu (x) = \lim_{n \to \infty} \int \frac{c^{n}x^{4n}}{n!} d\mu (x) =  \lim_{n \to \infty} \int \frac{c^{n}x^{4n}}{n!} w(x)d_{q}x = \int e^{c(x^{4})} w(x)d_{q}x \,.
\end{equation}
The support of $w(x)$ is compact, therefore
\[ \int e^{c(x^{2k})} d\mu (x) = \int e^{c(x^{2k})} w(x)d_{q}x < \infty \,. \]
Applying Lemma \ref{uniqueness} we find
\[ \int x^{2} d\mu (x) = \int x^{2} w(x)d_{q}x \,. \]
Taking the integral with respect to $\mu$ of Equation \eqref{recurrence2}, when $n=1$, we see that
\[ \int x^{2} d\mu (x) = a_{1} \,. \]
Thus, $a_{1} = \int x^{2} w(x)d_{q}x$.
\end{proof}

\section{Results and Discussion}\label{discussion}
The main result of this paper is Theorem \ref{final theorem}, the proof can be split into two parts. 
\begin{enumerate}
    \item First, we show that if the recurrence coefficients of a set of monic polynomials, $\{P_{n}\}_{n=0}^{\infty}$, satisfy Equation \eqref{even-odd}, then $\int x^{4j} d\mu$ is independent of the measure $\mu$ for all $j$ in $\mathbb{N}$, where $\mu$ is a positive measure such that $\int P_{i}P_{k} d\mu = h_{i}\delta_{i,k}$. \\ 
    \item Second, we show that if for some $p$ in $\mathbb{N}$, $m_{j} = \int x^{2pj} d\mu$ is known for all $j$ and decays sufficiently fast, then $\int x^{2} d\mu$ is uniquely determined by $\{m_{j}\}_{j=0}^{\infty}$.
\end{enumerate} 
The same method of proof applies to other non-linear discrete equations. For example, the method can be applied to continuous orthogonal polynomials. Equation \eqref{intro eq 2} can be shown to have a unique positive solution by directly applying the theory presented in this paper to the continuous case. With a slight modification to the first step of the proof, the existence of a unique positive solution to higher order, continuous and $q$-discrete, non-linear discrete equations can also be shown. There is also the possibility of applying the presented proof to non-linear discrete equations which do not arise from orthogonal polynomials. \\ \\
To enable an extension, an interesting direction would be the long open question, what subset of moments are needed to show a measure is unique? This would enable a proof of similar results for discrete equations satisfied by the recurrence coefficients of polynomials orthogonal to measures of the form $|x|^{\kappa}e^{P(x)}dx$, where $P(x)$ is a polynomial of even degree with negative leading order. Likewise, it would enable a proof for the $q$-discrete analogue, and possibly the $d$-discrete case. In particular, such an extension would confirm a similar conjecture by Boelen \textit{et al.} \cite{Boelen} for a generalised form of Equation \eqref{even-odd}, which remains an open problem.  

\section*{Acknowledgement(s)}
I would like to thank my supervisor Prof. Nalini Joshi for suggesting this field of research and highlighting useful resources, as well as her invaluable help in writing the paper. 

\section*{Funding}
This research was supported by an Australian Government Research Training Program (RTP) Scholarship and by the University of Sydney under the Postgraduate Research Supplementary Scholarship in Integrable Systems.

\bibliographystyle{plain}
\bibliography{References}

\end{document}